\documentclass[12pt]{article}
\usepackage{amsmath}
\usepackage{paralist}
\usepackage{amsthm}
\usepackage{amssymb}
\usepackage{esint}
\usepackage{dina42e}

 \usepackage[colorlinks=true]{hyperref}

\newtheorem{theorem}{Theorem}[section]

\newtheorem{lemma}[theorem]{Lemma}
\newtheorem{proposition}{Proposition}

\newtheorem{definition}[theorem]{Definition}

\newcommand{\N}{\mathbb{N}}
\newcommand{\R}{\mathbb{R}}
\newcommand{\bpsi}{\boldsymbol{\psi}}
\newcommand{\sd}{\, d}
\newcommand{\btau}{{\boldsymbol{\tau}}}
\newcommand{\no}{\mathbf{n}}
\newcommand{\dist}{\operatorname{dist}}
\newcommand{\sdist}{\operatorname{sdist}}
\newcommand{\G}{\Gamma}
\newcommand{\Gmu}[1]{\Gamma_{\mu,#1}}
\newcommand{\Gv}[1]{\Gamma_{S,#1}}
\newcommand{\bv}{\mathbf{v}}

\begin{document}
\title{On a linearized Mullins-Sekerka/Stokes system for two-phase flows}

\author{Helmut Abels\thanks{  \textit{Fakult\"at f\"ur Mathematik,   Universit\"at Regensburg,   93040 Regensburg,   Germany}   \textsf {helmut.abels@ur.de} } \  and Andreas Marquardt\thanks{\textit   {Fakult\"at f\"ur Mathematik,   Universit\"at Regensburg,   93040 Regensburg,   Germany}  } }
\maketitle
\begin{abstract}
   We study a linearized Mullins-Sekerka/Stokes system in a bounded domain with various boundary conditions. This system plays an important role to prove the convergence of a Stokes/Cahn-Hilliard system to its sharp interface limit, which is a Stokes/Mullins-Sekerka system, and to prove solvability of the latter system locally in time. We prove solvability of the linearized system in suitable $L^2$-Sobolev spaces with the aid of a maximal regularity result for non-autonomous abstract linear evolution equations.
\end{abstract}
{\small\noindent
{\textbf {Mathematics Subject Classification (2000):}}
Primary: 76T99; Secondary: 35Q30, 
35Q35, 
35R35,
76D05, 
76D45.}\\
{\textbf {Key words:}} {Two-phase flow, sharp interface limit, Cahn-Hilliard equation, Free boundary problems, Mullins-Sekerka equation,}

\section{Introduction}
We study the following linearized Mullins-Sekerka/Stokes system
\begin{align*}
D_{t,\Gamma}h+\mathbf{b}\cdot\nabla_{\Gamma}h-bh+\tfrac{1}{2}X_{0}^{*}\big((\mathbf{v}^{+}+\mathbf{v}^{-})\cdot\mathbf{n}_{\Gamma_{t}}\big)+\tfrac{1}{2}X_{0}^{*}\big(\big[\partial_{\mathbf{n}_{\Gamma_{t}}}\mu\big]\big) & =g &  & \text{on }\Sigma\times\left(0,T\right),\\
h\left(.,0\right) & =h_{0} &  & \text{in }\Sigma,
\end{align*}
where for every $t\in\left[0,T\right]$, the functions {$\mathbf{v}^{\pm}=\mathbf{v}^{\pm}(x,t)$},
$p^{\pm}=p^{\pm}(x,t)$ and $\mu^{\pm}=\mu^{\pm}(x,t)$
for $(x,t)\in\Omega_{T}^{\pm}$ with $\mathbf{v}^{\pm}\in H^{2}(\Omega^{\pm}(t))^d$,
$p^{\pm}\in H^{1}(\Omega^{\pm}(t))$ and $\mu^{\pm}\in H^{2}(\Omega^{\pm}(t))$
are the unique solutions to 
\begin{align}
\Delta\mu^{\pm} & =a_{1} &  & \text{in }\Omega^{\pm}(t),\label{eq:laplmu}\\
  \mu^{\pm} & = X_{0}^{*,-1}\big(\sigma\Delta_{\Gamma}h\pm a_{2}h\big)+a_{3} &  & \text{on }\Gamma_{t},\\
  \no\cdot \nabla \mu^{-} & =a_{4} &  & \text{on }\Gmu{1},\\
  \mu^{-} & =a_{4} &  & \text{on }\Gmu{2},\\
-\Delta\mathbf{v}^{\pm}+\nabla p^{\pm} & =\mathbf{a}_{1} &  & \text{in }\Omega^{\pm}(t),\label{eq:stokes}\\
\operatorname{div}\mathbf{v}^{\pm} & =0 &  & \text{in }\Omega^{\pm}(t),\\
[\mathbf{v}] & =\mathbf{a}_{2} &  & \text{on }\Gamma_{t},\label{eq:stokesjmp}\\
\left[2D_{s}\mathbf{v}-p\mathbf{I}\right]\mathbf{n}_{\Gamma_{t}} & =X_{0}^{*,-1}\big(\mathbf{a}_{3}h+\mathbf{a}_{4}\Delta_{\Gamma}h+a_{5}\nabla_{\Gamma}h+\mathbf{a}_{5}\big) &  & \text{on }\Gamma_{t},\label{eq:stokesbdry}\\
B_j(\mathbf{v}^{-},p^{-})& =\mathbf{a}_{6} &  & \text{on }\Gv{j}, j=1,2,3.\label{eq:stokesoutbdry}
\end{align}
Here $\Omega\subseteq \R^d$, $d=2,3$, is a bounded domain with smooth boundary, which is the disjoint union of $\Omega^+(t)$, $\Omega^-(t)$ and $\Gamma_t$, where $\Gamma_t=\partial\Omega^+(t)$ is a smoothly evolving $(d-1)$-dimensional orientable hypersurface. We assume that $\Gamma_t\subseteq\Omega$ for all $t\in (0,T)$, i.e., there is no boundary contact and contact angle. Moreover, $\Gamma_t$ is given for $t\in [0,T]$ as well as $a_1,\ldots, a_4$, $\mathbf{a}_1,\ldots, \mathbf{a}_6$ are given for some $T>0$, $\sigma>0$ is the surface tension constant and $D_s\mathbf{v}= \frac12(\nabla \mathbf{v}+\nabla \mathbf{v}^T)$. Furthermore,  
\begin{align*}
\left[g\right]\left(p,t\right) & :=\lim_{h\searrow0}\left[g^+(p+\mathbf{n}_{\Gamma_{t}}(p)h,t)-g^-(p-\mathbf{n}_{\Gamma_{t}}(p)h,t)\right]\text{ for }p\in\Gamma_{t}
\end{align*}
for suitable functions $g^\pm$ and $X_0\colon \Sigma\times [0,T]\to \Gamma:=\bigcup_{t\in [0,T]} \Gamma_t\times \{t\}$ is a suitable diffeomorphism, which is described in Section~\ref{sec:Prelim} below. For $a\colon \Sigma\times [0,T]\to \R^N$, $N\in\N$, we define $X^{\ast,-1}_0 a\colon \Gamma \to \R^N$ by
\begin{equation*}
  (X^{\ast,-1}_0 a)(p,t)= a(X_0^{-1}(p,t))\qquad \text{for all }(p,t)\in \Gamma
\end{equation*}
and for $b\colon \Gamma\to \R^N$, $N\in\N$, we define $X^{\ast}_0 b\colon \Sigma\times [0,T] \to \R^N$ by
\begin{equation*}
  (X^{\ast}_0 b)(s,t)= b(X_0(s,t))\qquad \text{for all }(s,t)\in \Sigma\times [0,T].
\end{equation*}

This system arises in the construction of approximate solutions in the proof of convergence of a Stokes/Cahn-Hilliard system to its sharp interface limit, which is a Stokes/Mullins-Sekerka system, cf.\ \cite{NSCH2}-\cite{NSCH1}. Here $\mathbf{v}^\pm\colon \bigcup_{t\in [0,T]}\Omega^\pm(t)\times \{t\}\to \R^d$ and $p^\pm\colon \bigcup_{t\in [0,T]}\Omega^\pm(t)\times \{t\}\to \R$ are the velocity and pressure incompressible viscous Newtonian fluids filling the domains $\Omega^\pm(t)$ at time $t$, which are separated by the (fluid) interface $\Gamma_t$. Furthermore, $h\colon \Sigma\times [0,T]\to \R$ is a linearized height function that describes the evolution of the interface at a certain order and $\mu^\pm\colon \bigcup_{t\in [0,T]}\Omega^\pm(t)\times \{t\}\to \R$ is a linearized chemical potential related to the fluids in $\Omega^\pm(t)$. If one neglects the terms related to $\mathbf{v}^\pm, p^\pm$, a similar linearized system arises in the study of the sharp interface limit of the Cahn-Hilliard equation, cf.\ \cite{abc}. Moreover, similar systems arise in the construction of strong solutions for a Navier-Stokes/Mullins-Sekerka system locally in time, cf. \cite{abelswilke}.

We consider different kinds of boundary conditions for $\mathbf{v}^-$ and $\mu^-$ simultaneously. More precisely, we assume that
\begin{equation*}
\partial\Omega= \Gmu{1}\cup\Gmu{2}= \Gv{1}\cup \Gv{2}\cup \Gv{3},  
\end{equation*}
where $\Gmu{1}, \Gmu{2}$ and $\Gv{1}, \Gv{2},\Gv{3}$ are disjoint and closed. Moreover, we have
\begin{alignat*}{2}
  B_1(\bv^-,p^-)&= \bv^- &\quad&\text{on }\Gv{1}\\
  (B_2(\bv^-,p^-))_{\btau}&= \left(\left(2D_s \bv^- -p^-\right)\no_{\partial\Omega}\right)_\btau + \alpha_2 \bv_\btau^- &\quad&\text{on }\Gv{2}\\
  \no_{\partial\Omega}\cdot B_2(\bv^-,p^-)&= \no_{\partial\Omega}\cdot \bv^- &\quad&\text{on }\Gv{2}\\
  B_3(\bv^-,p^-)&= \left(2D_s \bv^- -p^-\right)\no_{\partial\Omega} + \alpha_3 \bv^- &\quad&\text{on }\Gv{3},
\end{alignat*}
where $\no_{\partial\Omega}$ denotes the exterior normal on $\partial\Omega$.
To avoid a non-trivial kernel in the following we assume that one of the following cases holds true:
\begin{equation*}
  |\Gv{1}|+ \alpha_2 |\Gv{2}|+\alpha_3|\Gv{3}|>0
\end{equation*}
Then Korn's inequality yields
\begin{equation}\label{eq:Korn}
  \|\bv\|_{H^1(\Omega)}\leq C \left(\|D_s\bv\|_{L^2(\Omega)} + \alpha_2 \|\bv_{\btau}\|_{L^2(\Gv{2})}+ \alpha_3 \|\bv\|_{L^2(\Gv{3})} \right)
\end{equation}
for all $\bv\in H^1(\Omega)^d$ with $\bv|_{\Gv{1}}=0$, $\no_{\partial\Omega}\cdot \bv|_{\Gv{2}}=0$, cf.\ \cite[Corollary 5.9]{korn}. 

The structure of this contribution is as follows: In Section~\ref{sec:Prelim} we summarize some preliminaries on the parametrization of the interface $\Gamma_t$ and non-autonomuous evolution equations.
In Section~\ref{sec:Main} we present and prove our main results on existence and smoothness of solutions to the linearized Mullins-Sekerka system. Finally, in the appendix we prove an auxilliary result on the existence of a pressure.

The results of this paper are extensions of results in the second author's PhD Thesis.

\section{Preliminaries}\label{sec:Prelim}
\subsection{Notation}

Throughout this manuscript we denote by $\xi\in C^{\infty}\left(\mathbb{R}\right)$ a
cut-off function such that
\begin{equation}
\xi(s)=1\text{ if }\left|s\right|\leq\delta,\,\xi(s)=0\text{ if }\left|s\right|>2\delta,\text{ and }0\geq s\xi'(s)\geq-4\ \text{if }\delta\leq\left|s\right|\leq2\delta.\label{eq:cut-off}
\end{equation}

\subsection{Coordinates}\label{subsec:Coordinates}

We parametrize $(\Gamma_t)_{t\in[0,T_0]}$ with the aid of a family of smooth diffeomorphisms $X_0\colon \Sigma\times [0,T_0]\to \Gamma= \bigcup_{t\in [0,T_0]}\Gamma_t\times \{t\}$. Here either $\Sigma\subseteq \R^d$ is a smooth $(d-1)$-dimensional compact, orientable manifold without boundary, where $d\geq 2$ is allowed, or $d=2$ and $\Sigma=\mathbb{T}^1$. We have included the latter case to cover the setting in \cite{nsac,NSCH2,NSCH1}. Moreover, $\no_{\Gamma_t}(x)$ denotes the exterior normal of $\Gamma_t$ in $x$ with respect to $\Omega^-(t)$ and we denote
\begin{equation*}
\no (s,t):= \no_{\Gamma_t}((X_0(s,t))_1)\qquad \text{for all}~ s\in \Sigma, t\in [0,T_0],
\end{equation*}
where $(X_0(s,t))_1\in \R^d$ denote the spatial components of $X_0(s,t)$.
In the following we will need a tubular neighborhood of $\Gamma_t$: For $\delta>0$ sufficiently small,  the orthogonal projection $P_{\Gamma_t}(x)$ of all
\begin{equation*}
x\in \Gamma_t(3\delta) =\{y\in \Omega: \dist(y,\Gamma_t)<3\delta\}
\end{equation*}
is well-defined and smooth. Moreover, we choose $\delta$ so small that $\dist(\partial\Omega,\Gamma_t)>3\delta$ for every $t\in [0,T_0]$. Every $x\in\Gamma_t(3\delta)$ has a unique representation
\begin{equation*}
x=P_{\Gamma_t}(x)+r\no_{\Gamma_t}(P_{\Gamma_t}(x))
\end{equation*}
 where $r=\sdist(\Gamma_t,x)$. Here
\begin{equation*}
  d_{\G}(x,t):=\sdist (\Gamma_t,x)=
  \begin{cases}
    \dist(\Omega^-(t),x) &\text{if } x\not \in \Omega^-(t),\\
    -\dist(\Omega^+(t),x) &\text{if } x \in \Omega^-(t).
  \end{cases}
\end{equation*}
For the following we define for $\delta'\in (0,3\delta]$
\begin{equation*}
  \Gamma(\delta') =\bigcup_{t\in [0,T_0]} \Gamma_t(\delta') \times\{t\}.
\end{equation*}

We introduce new coordinates in  $\Gamma(3\delta)$ which we denote by
\begin{equation*}
  X\colon  (-3\delta, 3\delta)\times \Sigma \times [0,T_0]\mapsto \Gamma(3\delta)~\text{by}~  X(r,s,t):= X_0(s,t)+r\no(s,t),
\end{equation*}
where
\begin{equation*}
  r=\sdist(\Gamma_t,x), \qquad s= (X_{0}^{-1}(P_{\Gamma_t}(x),t))_1=: S(x,t),
\end{equation*}
where $(X_{0}^{-1}(P_{\Gamma_t}(x),t))_1$ denote the components in $\Sigma$ of $X_{0}^{-1}(P_{\Gamma_t}(x),t)$.

 In the case that $h$ is twice continuously differentiable with respect to $s$ and continuously differentiable with respect to $t$,
we introduce the notations 
 \begin{align}
D_{t,\Gamma}h(s,t) & :=\left.\partial_t\left(h(S(x,t),t)\right)\right|_{x=X_0(s,t)},\quad
\nabla_{\Gamma}h(s,t) :=\left.\nabla\left(h(S(x,t),t)\right)\right|_{x=X_0(s,t)},\nonumber \\
\Delta_{\Gamma}h(s,t)& :=\left.\Delta\left(h(S(x,t),t)\right)\right|_{x=X_0(s,t)},\nonumber
 \end{align}
 where $\nabla$ and $\Delta$ act with respect to $x$.
We note that in the case that $d=2$ and $\Sigma=\mathbb{T}^1$ we have
 \begin{align}
D_{t,\Gamma}h(s,t) & =\left(\partial_{t}+\partial_{t}S(X_0(s,t))\cdot \partial_{s}\right)h(s,t),\nonumber \\
\nabla_{\Gamma}h(s,t)& =\nabla S\left(X_0(s,t)\right)\partial_{s}h(s,t),\nonumber \\
\Delta_{\Gamma}h(s,t)& =\Delta S\left(X_0(s,t)\right) \partial_{s}h(s,t)+|\nabla S\left(X_0(s,t)\right)|^2\partial_{s}^2h(s,t).\nonumber
 \end{align}
 as in \cite{nsac,NSCH2, NSCH1}.

\subsection{Maximal Regularity for Non-autonomous Equations}

In order to prove our main result we use of the theory of maximal
regularity for non-autonomous abstract evolution equations. Therefore, we give a short overview of the basic definitions
and results which we will use. These are taken from \cite{arendt}
and all the proofs of the statements can be found in that article. 

In this subsection let $X$ and $D$ be two Banach spaces such that $D$
is continuously and densely embedded in $X$.
\begin{definition}[$L^{p}$-maximal regularity]
 Let $p\in\left(1,\infty\right).$
\begin{enumerate}
\item Let $A\in\mathcal{L}\left(D,X\right)$. Then $A$ has $L^{p}-$\emph{maximal
regularity} and we write $A\in\mathcal{MR}_{p}$ if for some bounded
interval $\left(t_{1},t_{2}\right)\subset\mathbb{R}$ and all $f\in L^{p}\left(t_{1},t_{2};X\right)$
there exists a unique $u\in W^{1,p}\left(t_{1},t_{2};X\right)\cap L^{p}\left(t_{1},t_{2};D\right)$
such that 
\begin{align*}
\partial_{t}u+Au & =f\quad\text{a.e. on }\left(t_{1},t_{2}\right),\\
u\left(t_{1}\right) & =0.
\end{align*}
\item Let $T>0$ and $A:\left[0,T\right]\rightarrow\mathcal{L}\left(D,X\right)$
be a bounded and strongly measurable function. Then $A$ has $L^{p}$-\emph{maximal
regularity} and we write $A\in\mathcal{MR}_{p}\left(0,T\right)$ if
for all $f\in L^{p}\left(0,T;X\right)$ there exists a unique $u\in W^{1,p}\left(0,T;X\right)\cap L^{p}\left(0,T;D\right)$
such that
\begin{align*}
\partial_{t}u+A\left(t\right)u & =f\quad\text{a.e. on }\left(0,T\right),\\
u\left(0\right) & =0.
\end{align*}
\end{enumerate}
\end{definition}

It can be shown that if $A\in\mathcal{MR}_{p}$ for some $p\in\left(1,\infty\right)$
then $A\in\mathcal{MR}_{p}$ for all $p\in\left(1,\infty\right)$.
Hence, we often simply write $A\in\mathcal{MR}$.
\begin{definition}[Relative Continuity]~\\
 We say that $A:\left[0,T\right]\rightarrow\mathcal{L}\left(D,X\right)$
is \emph{relatively continuous} if for each $t\in\left[0,T\right]$
and all $\epsilon>0$ there exist $\delta>0$, $\eta\geq0$ such that
for all $x\in D$ and for all $s\in\left[0,T\right]$ with $\left|s-t\right|\leq\delta$
the inequality 
\[
\left\Vert A(t)x-A(s)x\right\Vert _{X}\leq\epsilon\left\Vert x\right\Vert _{D}+\eta\left\Vert x\right\Vert _{X}
\]
holds.
\end{definition}

\begin{theorem}
\label{maxres} Let $T>0$ and $A:\left[0,T\right]\rightarrow\mathcal{L}\left(D,X\right)$
be a strongly measurable and relatively continuous function. If $A(t)\in\mathcal{MR}$
for all $t\in\left[0,T\right]$, then $A\in\mathcal{MR}_{p}\left(0,t\right)$
for every $0<t\leq T$ and every $p\in\left(1,\infty\right)$.
\end{theorem}

\begin{proof}
See \cite[Theorem 2.7]{arendt}.
\end{proof}
A very important tool for proving maximal regularity properties of
differential operators are perturbation techniques. Employing these can
often help to show maximal regularity for a variety of operators by
separating them into a main part (for which maximal regularity can
be readily shown) and a perturbation. 

In the following we give a perturbation result which is key to many
results in the next chapter.
\begin{definition}[Relatively Close]~\\
\label{relclo} Let $Y$ be a Banach space such that 
\[
D\hookrightarrow Y\hookrightarrow X.
\]
We say $Y$ \emph{is close to $X$ compared with $D$}, if for each
$\epsilon>0$ there exists $\eta\geq0$ such that 
\[
\left\Vert x\right\Vert _{Y}\leq\epsilon\left\Vert x\right\Vert _{D}+\eta\left\Vert x\right\Vert _{X}\quad\text{for all }x\in D.
\]
\end{definition}

\begin{proposition}
\label{relres} Let $Y$ be as in Definition \ref{relclo} and let
the inclusion $D\hookrightarrow Y$ be compact. Then $Y$
is close to $X$ compared with $D$.
\end{proposition}

\begin{proof}
See \cite[Example 2.9 (d)]{arendt}.
\end{proof}
\begin{theorem}\label{thm:Perturbation}
\label{perres} Let $T>0$ and $Y$ be a Banach space that is close
to $X$ compared with $D$. Furthermore, let $A\colon \left[0,T\right]\rightarrow\mathcal{L}\left(D,X\right)$
be relatively continuous and $B\colon \left[0,T\right]\rightarrow\mathcal{L}\left(Y,X\right)$
be strongly measurable and bounded. If $A\left(t\right)\in\mathcal{MR}$
for every $t\in\left[0,T\right]$ then $A+B\in\mathcal{MR}_{p}\left(0,T\right)$.
\end{theorem}

\begin{proof}
See \cite[Theorem 2.11]{arendt}.
\end{proof}

\section{Main Results}\label{sec:Main}

\subsection{Parabolic Equations on Evolving Surfaces \label{sec:Parabolic-Equations-on}}

We introduce the space 
\begin{eqnarray}
X_{T} & = & L^{2}\big(0,T;H^{\frac{7}{2}}(\Sigma)\big)\cap H^{1}\big(0,T;H^{\frac{1}{2}}(\Sigma)\big)\label{eq:XT}
\end{eqnarray}
for $T\in(0,\infty) $, where we equip
$X_{T}$ with the norm 
\[
\left\Vert h\right\Vert _{X_{T}}=\left\Vert h\right\Vert _{L^{2}\big(0,T;H^{\frac{7}{2}}(\Sigma)\big)}+\left\Vert h\right\Vert _{H^{1}\big(0,T;H^{\frac{1}{2}}(\Sigma)\big)}+\left\Vert h|_{t=0}\right\Vert _{H^{2}(\Sigma)}.
\]



\begin{theorem}
\label{Max-Reg} Let $T\in\left(0,T_{0}\right]$. Let $\mathbf{b}\colon \Sigma\times[0,T]\rightarrow\mathbb{R}^d$
and $b_{1},b_{2}\colon \Sigma\times[0,T]\rightarrow\mathbb{R}$
be smooth given functions.
 For every $g\in L^{2}\big(0,T;H^{\frac{1}{2}}(\Sigma)\big)$
and $h_{0}\in H^{2}(\Sigma)$, there is a unique
solution $h\in X_{T}$ of 
\begin{align}
D_{t,\Gamma}h+\mathbf{b}\cdot\nabla_{\Gamma}h-b_{1}h+X_{0}^{*}\big(\big[\partial_{\mathbf{n}_{\Gamma_{t}}}\mu\big]\big) & =g &  & \text{ on }\Sigma\times(0,T),\label{eq:hmax}\\
h\left(.,0\right) & =h_{0} &  & \text{ on }\Sigma,\nonumber 
\end{align}
where $\mu|_{\Omega^{\pm}(t)}\in H^{2}(\Omega^{\pm}(t))$,
for $t\in[0,T]$, is determined by \begin{subequations}\label{eq:musystem}
\begin{align}
\Delta\mu^{\pm} & =0 &  & \text{in }\Omega^{\pm}(t),\label{eq:musystem1}\\
\mu^{\pm} & =X_{0}^{*,-1}(\sigma\Delta_{\Gamma}h\pm b_{2}h) &  & \text{on }\Gamma_{t},\label{eq:musystem2}\\
  \no_{\partial\Omega}\cdot \nabla \mu^{-} & =0 &  & \text{on }\Gmu{1},\\
  \mu^{-} & =0 &  & \text{on }\Gmu{2}.\label{eq:musystem3}
\end{align}
\end{subequations}Furthermore, the estimates 
\begin{align}
\sum_{\pm}\left\Vert \mu^{\pm}\right\Vert _{L^{2}\left(0,T;H^{2}\left(\Omega^{\pm}(t)\right)\right)\cap L^{6}\left(0,T;H^{1}\left(\Omega^{\pm}(t)\right)\right)} & \leq C\left\Vert h\right\Vert _{X_{T}},\label{eq:mumaxab}
\end{align}
hold for some constant $C>0$ independent of $\mu$ and $h$.
\end{theorem}

\begin{proof}
We may write (\ref{eq:hmax}) in abstract form as 
\begin{alignat*}{2}
\partial_{t}h+\mathcal{A}(t)h & =g &\qquad & \text{in }\Sigma\times [0,T],\\
h\left(.,0\right) & =h_{0} &  & \text{in }\Sigma,
\end{alignat*}
where $\mathcal{A}(t)\colon  H^{\frac{7}{2}}(\Sigma)\to  H^{\frac{1}{2}}(\Sigma)$ depends on  $t\in[0,T]$.
Now we fix $t_{0}\in [0,T]$ and analyze the operator $\mathcal{A}(t_{0})$,
where we replace $t$ with the fixed $t_{0}$ in all time dependent
coefficients. 

In order to understand this operator we define 
\begin{align*}
\mathfrak{D}_{t_{0}}\colon & H^{\frac{7}{2}}(\Sigma)\rightarrow H^{\frac{3}{2}}(\Gamma_{t_{0}})\colon  h\mapsto\big(X_{0}^{*,-1}(\sigma\Delta_{\Gamma}h)\big)(.,t_{0}),\\
S_{t_{0}}^{N}\colon & H^{\frac{3}{2}}(\Gamma_{t_{0}})\rightarrow H^{2}(\Omega^{+}(t_{0}))\times H^{2}(\Omega^{-}(t_{0}))\colon f\mapsto(\Delta_{N})^{-1}f,\\
B_{t_{0}}\colon & H^{2}(\Omega^{+}(t_{0}))\times H^{2}(\Omega^{-}(t_{0}))\rightarrow H^{\frac{1}{2}}(\Sigma)\colon (\mu^{+},\mu^{-})\mapsto \big(X_{0}^{*}([\nabla\mu\cdot\mathbf{n}_{\Gamma_{t_{0}}}])\big)(.,t_{0}),
\end{align*}
where $(\Delta_{N})^{-1}f$ is the
unique solution $(\mu_{N}^{+},\mu_{N}^{-})$ to \begin{subequations}\label{eq:neumannelliptic}
\begin{align}
\Delta\mu_{N}^{\pm} & =0 &  & \text{in }\Omega^{\pm}(t_{0}),\label{eq:neumannelliptic1}\\
\mu_{N}^{\pm} & =f &  & \text{on }\Gamma_{t_{0}},\label{eq:neumannelliptic2}\\
\nabla\mu_{N}^{-}\cdot\mathbf{n}_{\partial\Omega} & =0 &  & \text{on }\partial\Omega.\label{eq:neumannelliptic3}
\end{align}
\end{subequations}
In the literature the concatenation $B_{t_{0}}\circ S_{t_{0}}^{N}$
is often referred to as the \emph{Dirichlet-to-Neumann} operator and
$A_{0}\left(t_{0}\right):=B_{t_{0}}\circ S_{t_{0}}^{N}\circ\mathfrak{D}_{t_{0}}$
is called the \emph{Mullins-Sekerka} operator. It can be shown that
\[
A_{0}\colon [0,T]\to\mathcal{L}\big(H^{\frac{7}{2}}(\Sigma),H^{\frac{1}{2}}(\Sigma)\big)
\]
has $L^{p}$-maximal regularity, i.e., $A_{0}\in\mathcal{MR}_{p}(0,T)$.
We will not prove this in detail but just give a short sketch describing
the essential ideas: first, a reference surface $\tilde\Sigma\subset\subset\Omega$
is fixed such that $\Gamma_{t}$ can be expressed as a graph over
$\tilde\Sigma$ for $t$ in some time interval $\left[\tilde{t},\tilde{t}+\epsilon\right]\subset\left[0,T\right]$.
e.g.\ one may choose $\tilde\Sigma:=\Gamma_{0}$ and then determine $\epsilon_{0}>0$
such that $\Gamma_{t}$ may be written as graph over $\Gamma_{0}$
for all $t\in\left[0,\epsilon_{0}\right]$, which is possible since
$\Gamma$ is a smoothly evolving hypersurface. Next, a Hanzawa transformation
is applied, enabling us to consider (\ref{eq:neumannelliptic3}) as
a system on fixed domains $\Omega^{\pm}$ and $\tilde\Sigma$, but with
time dependent coefficients (see e.g.\ \cite[Chapter 2.2]{abelswilke} or and \cite[Chapter 4]{stefan}). Here, $\Omega^{+}$, $\Omega^{-}$
and $\tilde\Sigma$ denote disjoint sets such that $\partial\Omega^{+}=\tilde\Sigma$
and $\Omega=\Omega^{+}\cup\Omega^{-}\cup\tilde\Sigma$ holds and we assume
in the following that $t_{0}\in\left[0,\epsilon_{0}\right]$. To be
more specific, the Hanzawa transformation results in a system of the
form 
\begin{align*}
a\left(x,t,\nabla_{x}\right)\bar{\mu}^{\pm} & =0 &  & \text{in }\Omega^{\pm},\\
\bar{\mu}^{\pm} & =\tilde{f} &  & \text{on }\tilde\Sigma,\\
\nabla\bar{\mu}^{-}\cdot\mathbf{n}_{\partial\Omega} & =0 &  & \text{on }\partial\Omega,
\end{align*}
where $a$ is the transformed Laplacian, depending smoothly on $t$
and $\tilde{f}$ is the transformation of $f$. Applying the Hanzawa
transformation (and the diffeomorphism $X_{0}$)
also to the operators $\mathfrak{D}_{t_{0}}$ and $B_{t_{0}}$, we
end up with a transformed operator $\tilde{A}_{0}(t_{0})\in\mathcal{L}\big(H^{\frac{7}{2}}(\Sigma),H^{\frac{1}{2}}(\Sigma)\big)$
and \cite[Corollary 6.6.5]{pruess} implies that $\tilde{A}_{0}(t_{0})$
has $L^{p}$-maximal regularity. As all involved differential operators
and coefficients depend smoothly on $t$, it is possible to show that
$\tilde{A}_{0}\colon \left[0,\epsilon_{0}\right]\to\mathcal{L}\big(H^{\frac{7}{2}}(\Sigma),H^{\frac{1}{2}}(\Sigma)\big)$
is relatively continuous. Therefore Theorem \ref{maxres} implies
$\tilde{A}_{0}\in\mathcal{MR}_{p}(0,\epsilon_{0})$ and,
transforming back, also $A_{0}\in\mathcal{MR}_{p}(0,\epsilon_{0})$.
Repeating this procedure with a new reference surface $\Sigma:=\Gamma_{\epsilon_{0}}$
and iteratively continuing the argumentation, we end up with $A_{0}\in\mathcal{MR}_{p}(0,T)$.


We proceed by showing that $\mathcal{A}(t_{0})=A_{0}(t_{0})+\mathcal{B}(t_{0})$
holds for some lower order perturbation $\mathcal{B}$. We introduce
\[
S_{t_{0}}^{DN}\colon H^{\frac{3}{2}}(\Gamma_{t_{0}})\rightarrow H^{2}(\Omega^{+}(t_{0}))\times H^{2}(\Omega^{-}(t_{0}))\colon f\mapsto (\Delta_{DN})^{-1}f,
\]
where $(\mu_{DN}^{+},\mu_{DN}^{-}):=(\Delta_{DN})^{-1} f$
is the unique solution to (\ref{eq:neumannelliptic}), replacing
$\nabla\mu_{N}^{-}\cdot\mathbf{n}_{\partial\Omega}=0$ on $\Gmu{2}$ by $\mu_{D}^{-}=0$
on $\Gmu{2}$. Moreover, we write $S_{t_{0}}^{\Delta}:=S_{t_{0}}^{DN}-S_{t_{0}}^{N}$
and observe that the equality 
\begin{equation}
B_{t_{0}}\circ S_{t_{0}}^{DN}\circ\mathfrak{D}_{t_{0}}=A_{0}(t_{0})+\mathcal{B}_{0}(t_{0})\label{eq:neumanndirichlet}
\end{equation}
is satisfied, where $\mathcal{B}_{0}(t_{0}):=B_{t_{0}}\circ S_{t_{0}}^{\Delta}\circ\mathfrak{D}_{t_{0}}$.
Let $f\in H^{\frac{3}{2}}(\Gamma_{t_{0}})$ be fixed, $(\mu_{DN}^{+},\mu_{DN}^{-}):=S_{t_{0}}^{DN}f$,
$(\mu_{N}^{+},\mu_{N}^{-}):=S_{t_{0}}^{N}f$
and $\tilde{\mu}^{\pm}:=\mu_{DN}^{\pm}-\mu_{N}^{\pm}$, implying $(\tilde{\mu}^{+},\tilde{\mu}^{-})=S_{t_{0}}^{\Delta}f$.
Then $\tilde{\mu}^{\pm}\in H^{2}(\Omega^{\pm}(t_{0}))$
solves
\begin{alignat*}{2}
\Delta\tilde{\mu}^{\pm} & =0 & \qquad & \text{in }\Omega^{\pm}\left(t_{0}\right),\\
\tilde{\mu}^{\pm} & =0 &  & \text{on }\Gamma_{t_{0}},\\
\no_{\partial\Omega}\cdot \nabla \tilde{\mu}^{-} & =0 &  & \text{on }\Gmu{1}\\
\tilde{\mu}^{-} & =\mu_{N}^{-} &  & \text{on }\Gmu{2}
\end{alignat*}
and elliptic regularity theory implies 
\begin{equation}
\left\Vert \tilde{\mu}^{-}\right\Vert _{H^{2}(\Omega^{-}(t_{0}))}\leq C\Vert \mu_{N}^{-}\Vert _{H^{\frac{3}{2}}(\Gmu{2})}\label{eq:mutildeelliptic}
\end{equation}
 and $\tilde{\mu}^{+}\equiv0$ in $\Omega^{+}(t_{0})$.
For the further argumentation, we show 
\begin{equation}
\left\Vert \mu_{N}^{-}\right\Vert _{H^{\frac{3}{2}}(\partial\Omega)}\leq C\Vert \mu_{N}^{-}\Vert _{H^{\frac{1}{2}}(\Gamma_{t_{0}})}.\label{eq:zwischenstand}
\end{equation}
To this end let $\gamma(x):=\xi(4d_{\mathbf{B}}(x))$
for all $x\in\Omega$, where $\xi$ is a cut-off function satisfying
(\ref{eq:cut-off}). In particular $\text{supp}\gamma\cap\Gamma_{t}=\emptyset$
for all $t\in[0,T_{0}]$ by our assumptions and $\gamma\equiv1$
in $\partial\Omega(\frac{\delta}{4})$. Denoting $\hat{\mu}:=\gamma\mu_{N}^{-}\in H^{2}\left(\Omega^{-}(t_{0})\right)$,
we compute using $\Delta\mu_{N}^{-}=0$ in $\Omega^{-}(t_{0})$
that $\hat{\mu}$ is a solution to
\begin{alignat*}{2}
\Delta\hat{\mu} & =2\nabla\gamma\cdot\nabla\mu_{N}^{-}+(\Delta\gamma)\mu_{N}^{-} &\qquad  & \text{in }\Omega^{-}(t_{0}),\\
\hat{\mu} & =0 &  & \text{on }\Gamma_{t_{0}},\\
\nabla\hat{\mu}\cdot\mathbf{n}_{\partial\Omega} & =0 &  & \text{on }\partial\Omega,
\end{alignat*}
which, again regarding elliptic regularity theory, implies $\left\Vert \hat{\mu}\right\Vert _{H^{2}(\Omega^{-}(t_{0}))}\leq C\Vert \mu_{N}^{-}\Vert _{H^{1}(\Omega^{-}(t_{0}))}$.
This is essential in view of (\ref{eq:zwischenstand}) as it leads
to
\begin{align*}
  \Vert \mu_{N}^{-}\Vert _{H^{\frac{3}{2}}(\partial\Omega)} & =\Vert \hat{\mu}\Vert _{H^{\frac{3}{2}}(\partial\Omega)}\leq C\Vert \hat{\mu}\Vert _{H^{2}(\Omega^{-}(t_{0}))}\\
  &\leq C\Vert \mu_{N}^{-}\Vert _{H^{1}(\Omega^{-}\left(t_{0})\right)} \leq \tilde{C}\Vert \mu_{N}^{-}\Vert _{H^{\frac{1}{2}}(\Gamma_{t_{0}})},
\end{align*}
where we used the continuity of the trace operator $\operatorname{tr}\colon H^{2}(\Omega^{-}(t_{0}))\rightarrow H^{\frac{3}{2}}(\partial\Omega^{-}(t_{0}))$
in the first inequality (cf.\ \cite[Theorem 3.37]{mclean})
and standard estimates for elliptic equations in the second and third
inequality. 

Let now $h\in H^{\frac{7}{2}}(\Sigma)$ and $(\tilde{\mu}^{+},\tilde{\mu}^{-}):=S_{t_{0}}^{\Delta}\circ\mathfrak{D}_{t_{0}} h$.
Our prior considerations enable us to estimate 
\begin{align*}
\left\Vert B_{t_{0}}\circ S_{t_{0}}^{\Delta}\circ\mathfrak{D}_{t_{0}}h\right\Vert _{H^{\frac{1}{2}}(\Sigma)} & \leq C\left\Vert \tilde{\mu}^{-}\right\Vert _{H^{2}(\Omega^{-}(t_{0}))}\le C\Vert \mu_{N}^{-}\Vert _{H^{\frac{3}{2}}(\partial\Omega)}\\
 & \leq C\Vert \mu_{N}^{-}\Vert _{H^{\frac{1}{2}}(\Gamma_{t_{0}})}\leq C\Vert \sigma\Delta_{\Gamma}h\Vert _{H^{\frac{1}{2}}(\Sigma)} \leq C\Vert h\Vert_{H^{\frac{5}{2}}(\Sigma)},
\end{align*}
where we employed the continuity of the trace in the first line, (\ref{eq:mutildeelliptic})
in the second, (\ref{eq:zwischenstand}) in the third and the definition
of $\mu_{N}^{-}$ in the fourth. As $H^{\frac{7}{2}}(\Sigma)$
is dense in $H^{\frac{5}{2}}(\Sigma)$, we may
extend $\mathcal{B}_{0}(t_{0})$ to an operator 
\begin{equation}
\mathcal{B}_{0}(t_{0})\colon H^{\frac{5}{2}}(\Sigma)\rightarrow H^{\frac{1}{2}}(\Sigma),\label{eq:B0}
\end{equation}
which shows in regard to (\ref{eq:neumanndirichlet}) that we may
view $B_{t_{0}}\circ S_{t_{0}}^{\Delta}\circ\mathfrak{D}_{t_{0}}$ as a lower order
perturbation of $A_{0}(t_0)$.

Next we take care of the term involving $b_{2}$ in (\ref{eq:musystem2}).
For this we consider the operator 
\[
\mathcal{B}_{1}(t_0)\colon H^{\frac{7}{2}}(\Sigma)\rightarrow H^{\frac{1}{2}}(\Sigma)\colon h\mapsto X_{0}^{*}\big(\big[\partial_{\mathbf{n}_{\Gamma_{t_{0}}}}\mu_{1}\big]\big),
\]
where $\mu_{1}^{\pm}\in H^{2}(\Omega^{\pm}(t_0))$
is the solution to 
\begin{alignat*}{2}
\Delta\mu_{1}^{\pm} & =0 &\qquad  & \text{in }\Omega^{\pm}(t_0),\\
\mu_{1}^{\pm} & =\pm b_{2}h &  & \text{on }\Gamma_{t_{0}},\\
\no_{\partial\Omega}\cdot \nabla \mu_{1}^{-} & =0 &  & \text{on }\Gmu{1},\\
\mu_{1}^{-} & =0 &  & \text{on }\Gmu{2}.
\end{alignat*}
We estimate
\begin{align}
\left\Vert \mathcal{B}_{1}(t_0)h\right\Vert _{H^{\frac{1}{2}}(\Sigma)} & \leq C\big\Vert \big[\partial_{\mathbf{n}_{\Gamma_{t_{0}}}}\mu_{1}\big]\big\Vert _{H^{\frac{1}{2}}\left(\Gamma_{t_{0}}\right)}\nonumber\\&\leq C\left(\Vert \mu_{1}^{+}\Vert _{H^{2}(\Omega^{+}(t_0))}+\Vert \mu_{1}^{-}\Vert _{H^{2}(\Omega^{-}(t_0))}\right)
  \leq C\Vert h\Vert _{H^{\frac{3}{2}}(\Sigma)},\label{eq:max-reg-per}
\end{align}
where $C>0$ can be chosen independent of $h$ and $t_{0}\in\left[0,T\right]$.
Here we again employed the continuity of the trace operator and elliptic
theory.

Defining 
\[
\mathcal{B}(t_0)\colon H^{\frac{7}{2}}(\Sigma)\to H^{\frac{1}{2}}(\Sigma)\colon h\mapsto\mathcal{B}(t_0)h:=\tilde{\mathbf{b}}(.,t_{0})\cdot \nabla_{\Gamma}h-b_{1}(.,t_{0})h+(\mathcal{B}_{0}(t_0)+\mathcal{B}_{1}(t_0))h,
\]
where $\tilde{\mathbf{b}}$ is chosen such that $\partial_th+\tilde{\mathbf{b}}\cdot \nabla_\Gamma h= D_{t,\Gamma}h+{\mathbf{b}}\cdot \nabla_\Gamma h $, and using (\ref{eq:max-reg-per})
and (\ref{eq:B0}), we find that
\[
\left\Vert \mathcal{B}(t_0)h\right\Vert _{H^{\frac{1}{2}}(\Sigma)}\leq C\Vert h\Vert _{H^{\frac{5}{2}}(\Sigma)}.
\]
Thus, we can extend $\mathcal{B}(t_0)$ to a bounded
operator $\mathcal{B}(t_0)\colon H^{\frac{5}{2}}(\Sigma)\rightarrow H^{\frac{1}{2}}(\Sigma)$.
Since $H^{\frac{5}{2}}(\Sigma)$ is close to $H^{\frac{1}{2}}(\Sigma)$
compared to $H^{\frac{7}{2}}(\Sigma)$ as the embedding $H^{\frac{7}{2}}(\Sigma)\hookrightarrow H^{\frac{5}{2}}(\Sigma)$
is compact, we get due to Theorem~\ref{thm:Perturbation}, that $\text{\ensuremath{\mathcal{A}}}=A_{0}+\mathcal{B}$
has $L^{p}$-maximal regularity for all $t\in\left[0,T\right]$.

By elliptic theory 
\begin{align*}
\left\Vert \mu^{\pm}\right\Vert _{H^{1}\left(\Omega^{\pm}(t)\right)} & \leq C\left\Vert X_{0}^{*,-1}\left(\sigma\Delta_{\Gamma}h+b_{2}h\right)\right\Vert _{H^{\frac{1}{2}}\left(\Gamma_{t}\right)}\leq C\left\Vert h\right\Vert _{H^{\frac{5}{2}}(\Sigma)}
\end{align*}
for almost all $t\in\left[0,T\right]$ and thus 
\[
\left\Vert \mu^{\pm}\right\Vert _{L^{6}\left(0,T;H^{1}\left(\Omega^{\pm}(t)\right)\right)}\leq C\left\Vert h\right\Vert _{L^{6}\left(0,T;H^{\frac{5}{2}}(\Sigma)\right)}\leq C\left\Vert h\right\Vert _{X_{T}}.
\]
\end{proof}
\begin{theorem}
\label{stokesthe} Let $T\in\left(0,T_{0}\right]$ and $t\in\left[0,T\right]$.
For every $\mathbf{f}\in L^{2}(\Omega)^d$, $\mathbf{s}\in H^{\frac{3}{2}}(\Gamma_{t})^d$, $\mathbf{a}\in H^{\frac{1}{2}}(\Gamma_{t})^d$
and $\mathbf{g}\colon \partial \Omega\to \R^d$ such that  $\mathbf{g}|_{\Gv{1}}\in H^{\frac{3}{2}}(\Gv{1})^d$, $\no_{\partial\Omega}\cdot \mathbf{g}|_{\Gv{2}}\in H^{\frac{3}{2}}(\Gv{2})$, $(I-\no_{\partial\Omega}\otimes\no_{\partial\Omega})\mathbf{g}|_{\Gv{2}}\in H^{\frac{1}{2}}(\Gv{2})^d$, $\mathbf{g}|_{\Gv{3}}\in H^{\frac{1}{2}}(\Gv{3})^d$ satisfying the compatibility condition
\begin{equation}
  \label{eq:compStokes}
  \int_{\Gamma_t} \no_{\Gamma_t}\cdot \mathbf{s} \, d\mathcal{H}^{d-1} + \int_{\partial\Omega} \no_{\partial\Omega}\cdot \mathbf{g} \, d\mathcal{H}^{d-1} =0\qquad \text{if }\Gv{3}=\emptyset
\end{equation}
the system 
\begin{align}
-\Delta\mathbf{v}^{\pm}+\nabla p^{\pm} & =\mathbf{f} &  & \text{in }\Omega^{\pm}(t),\label{eq:zweiphasen1}\\
\mathrm{div}\mathbf{v}^{\pm} & =0 &  & \text{in }\Omega^{\pm}(t),\\
B_j(\mathbf{v}^{-},p^{-}) & =\mathbf{g}|_{\Gv{j}}=: \mathbf{g}_j &  & \text{on }\Gv{j}, j=1,2,3,\\
\left[\mathbf{v}\right] & =\mathbf{s} &  & \text{on }\Gamma_{t},\\
\left[2D_{s}\mathbf{v}-p\mathbf{I}\right]\mathbf{n}_{\Gamma_{t}} & =\mathbf{a} &  & \text{on }\Gamma_{t}\label{eq:zweiphasen5}
\end{align}
has a unique solution $\left(\mathbf{v}^{\pm},p^{\pm}\right)\in H^{2}(\Omega^{\pm}(t))^d\times H^{1}(\Omega^{\pm}(t))$ satisfying $\int_\Omega p\, dx=0$ if $\Gv{3}=\emptyset$.
Moreover, there is a constant $C>0$ independent of $t\in\left[0,T_{0}\right]$
such that
\begin{align}\nonumber
  &\left\Vert \left(\mathbf{v},p\right)\right\Vert _{H^{2}(\Omega^{\pm}(t))\times H^{1}(\Omega^{\pm}(t))}\leq C \left(\| \mathbf{f}\| _{L^{2}(\Omega)}+\left\Vert \mathbf{s}\right\Vert _{H^{\frac{3}{2}}(\Gamma_{t})}+\left\Vert \mathbf{a}\right\Vert _{H^{\frac{1}{2}}(\Gamma_{t})}\right. \\
  &\qquad \left.+\left\Vert \mathbf{g}_1\right\Vert _{H^{\frac{3}{2}}(\Gv{1})}+ \left\Vert \mathbf{g}_{2,\no}\right\Vert _{H^{\frac{3}{2}}(\Gv{2})} + \left\Vert \mathbf{g}_{2,\btau}\right\Vert _{H^{\frac{1}{2}}(\Gv{2})} +\left\Vert \mathbf{g}_3\right\Vert _{H^{\frac{1}{2}}(\Gv{3})}\right)\label{eq:instoab}
\end{align}
holds.
\end{theorem}
\begin{proof}
 We can assume for simplicity that $\mathbf{g}=0$ on $\Gv{1}$ and $\no_{\partial\Omega}\cdot \mathbf{g}=0$ on $\Gv{2}$. Otherwise we substract a suitable extension of $\mathbf{g}$.
As a first step, we reduce the system (\ref{eq:zweiphasen1})\textendash (\ref{eq:zweiphasen5})
to the case $\mathbf{s}=0$. Elliptic theory implies that the equation
\begin{align*}
\Delta q & =0 &  & \text{in }\Omega^{-}(t),\\
\nabla q\cdot\mathbf{n}_{\Gamma_{t}} & =\mathbf{s}\cdot\mathbf{n}_{\Gamma_{t}} &  & \text{on }\Gamma_{t},\\
  \no_{\partial\Omega} \cdot \nabla q & =\no_{\partial\Omega}\cdot \mathbf{g}_{2} &  & \text{on }\Gv{1}\cup \Gv{2},\\
  q & =0 &  & \text{on }\Gv{3}
\end{align*}
has a unique solution $q\in H^{3}\left(\Omega^{-}(t)\right)$ with $\int_{\Omega^-(t)} q\, dx =0$ if $\Gv{3}=\emptyset$
since $\mathbf{s}\in H^{\frac{3}{2}}\left(\Gamma_{t}\right)^d$ and $\no_{\partial\Omega}\cdot \mathbf{g}_{2}\in H^{\frac32}(\Gv{1}\cup \Gv{2})$. Here, if $\Gv{3}=\emptyset$, the necessary compatibility condition is satisfied because of \eqref{eq:compStokes}. Moreover,
we have the estimate 
\[
\left\Vert q\right\Vert _{H^{3}\left(\Omega^{-}(t)\right)}\leq C\left(\left\Vert \mathbf{s}\right\Vert _{H^{\frac{3}{2}}(\Gamma_{t})}+ \|\no_{\partial\Omega}\cdot \mathbf{g}_{2}\|_{H^{\frac32}(\Gv{1}\cup \Gv{2})}\right).
\]
Regarding the tangential part of $\mathbf{s}$, we may solve the stationary
Stokes system 
\begin{align*}
-\Delta\mathbf{w}+\nabla\tilde{p} & =0 &  & \text{in }\Omega^{-}(t),\\
\operatorname{div}\mathbf{w} & =0 &  & \text{in }\Omega^{-}(t),\\
  \mathbf{w} & =(I-\no_{\Gamma_t}\otimes \no_{\Gamma_t})\left(\mathbf{s}-\nabla q\right) &  & \text{on }\Gamma_{t},\\
  \mathbf{w} & =0 &  & \text{on }\partial\Omega.
\end{align*}
We may find a solution $\left(\mathbf{w},\tilde{p}\right)\in H^{2}(\Omega^{-}(t))^d\times H^{1}(\Omega^{-}(t))$
(made unique by the normalization $\int_{\Omega^{-}(t)}\tilde{p}\sd x=0$)
and also get the estimate 
\[
\left\Vert \mathbf{w}\right\Vert _{H^{2}(\Omega^{-}(t))}+\left\Vert \tilde{p}\right\Vert _{H^{1}(\Omega^{-}(t))}\leq C\left(\left\Vert \mathbf{s}\right\Vert _{H^{\frac{3}{2}}(\Gamma_{t})}+ \|\no_{\partial\Omega}\cdot \mathbf{g}_{2}\|_{H^{\frac32}(\Gv{1}\cup \Gv{2})}\right).
\]
 Thus, defining 
$
\tilde{\mathbf{w}}:=\mathbf{w}+\nabla q, 
$
the couple $\left(\tilde{\mathbf{w}},\tilde{p}\right)$ solves 
\begin{align*}
-\Delta\tilde{\mathbf{w}}+\nabla\tilde{p} & =0 &  & \text{in }\Omega^{-}(t),\\
\operatorname{div}\tilde{\mathbf{w}} & =0 &  & \text{in }\Omega^{-}(t),\\
\tilde{\mathbf{w}} & =\mathbf{s} &  & \text{on }\Gamma_{t},\\
\tilde{\mathbf{w}} & =0 &  & \text{on }\partial\Omega,
\end{align*}
and may be estimated by $\mathbf{s}$ in strong norms. Next, let 
\[
  \tilde{\mathbf{g}}:=\mathbf{g}_j+B_j(\tilde{\mathbf{w}},\tilde{p})\qquad \text{on }\Gv{j}, j=1,2,3
\]
and $\tilde{\mathbf{a}}:=\mathbf{a}-\left(2D_{s}\tilde{\mathbf{w}}-\tilde{p}\mathbf{I}\right)\mathbf{n}_{\Gamma_t}\in H^{\frac{1}{2}}(\Gamma_{t})^d$,
where the regularity is due to the properties of the trace operator.
Then, for every strong solution $\left(\hat{\mathbf{v}}^{\pm},\hat{p}^{\pm}\right)$
of \eqref{eq:zweiphasen1}-\eqref{eq:zweiphasen5}, with
$\mathbf{s\equiv}0$ and $\mathbf{g},\mathbf{a}$ substituted by $\tilde{\mathbf{g}},\tilde{\mathbf{a}}$,
the functions
\[
\left(\mathbf{v}^{+},p^{+}\right):=\left(\hat{\mathbf{v}}^{+},\hat{p}^{+}\right)\text{ and }\left(\mathbf{v}^{-},p^{-}\right):=\left(\hat{\mathbf{v}}^{-}-\tilde{\mathbf{w}},\hat{p}^{-}-\tilde{p}\right)
\]
are solutions to the original system \eqref{eq:zweiphasen1}-\eqref{eq:zweiphasen5}.
So, we will consider $\mathbf{s}\equiv0$ in the following and show
existence of strong solutions in that case.

As a starting point, we construct a solution $\left(\mathbf{v},p\right)\in V(\Omega)\times L^{2}(\Omega)$
to the weak formulation 
\begin{align}
  &\int_{\Omega}2D_{s}\mathbf{v}:D_{s}\bpsi\, dx -\int_\Omega p\operatorname{div}\bpsi\sd x+\int_{\Gv{2}}\alpha_{2}\mathbf{v}\cdot\bpsi\sd\mathcal{H}^{d-1}(s)+\int_{\Gv{3}}\alpha_{3}\mathbf{v}\cdot\bpsi\sd\mathcal{H}^{d-1}(s)\nonumber\\
  &\qquad =\int_{\Omega}\mathbf{f}\cdot\bpsi\sd x+\int_{\Gamma_{t}}\mathbf{a}\cdot\bpsi\sd\mathcal{H}^{d-1}(s)-\int_{\partial\Omega}\mathbf{g}\cdot\bpsi\sd\mathcal{H}^{d-1}(s),\label{eq:weak2phase}
\end{align}
for all $\bpsi\in H^{1}(\Omega)^d$ with $\bpsi|_{\Gv{1}}=0$, $\no \cdot \bpsi|_{\Gv{2}}=0$, where
\begin{equation*}
  V(\Omega)= \{\mathbf{u}\in H^1(\Omega)^d: \operatorname{div} \mathbf{u}=0, \mathbf{u}|_{\Gv{1}}=0, \no \cdot \mathbf{u}|_{\Gv{2}}=0\}.
\end{equation*}
Considering first $\bpsi\in V(\Omega)$
and the right hand side as a functional $\mathbf{F}\in\left(V(\Omega)\right)'$,
the Lemma of Lax-Milgram implies the existence of a unique $\mathbf{v}\in V(\Omega)$
solving (\ref{eq:weak2phase}) for all $\bpsi\in V(\Omega)$,
where the coercivity of the involved bilinear form is a consequence
of \eqref{eq:Korn}.

Next consider the functional
\begin{align*}
 F(\bpsi):= &-\int_{\Omega}2D_{s}\mathbf{v}:D_{s}\bpsi\, dx -\int_{\Gv{2}}\alpha_{2}\mathbf{v}\cdot\bpsi\sd\mathcal{H}^{d-1}(s)-\int_{\Gv{3}}\alpha_{3}\mathbf{v}\cdot\bpsi\sd\mathcal{H}^{d-1}(s)\\
  &\qquad +\int_{\Omega}\mathbf{f}\cdot\bpsi\sd x+\int_{\Gamma_{t}}\mathbf{a}\cdot\bpsi\sd\mathcal{H}^{d-1}(s)-\int_{\partial\Omega}\mathbf{g}\cdot\bpsi\sd\mathcal{H}^{d-1}(s),
\end{align*}
for all $\bpsi\in H^1(\Omega)^d$ with $\bpsi|_{\Gv{1}}=0$, $\no \cdot \bpsi|_{\Gv{2}}=0$. Then $F$ vanishes on $V(\Omega)$ and by Lemma~\ref{lem:pressure} in  Appendix A 
there is a unique $p\in L^2 (\Omega)$ with  $\int_\Omega p\, dx =0$ if $\Gv{3}=\emptyset$ such that
\begin{equation*}
  F(\bpsi)= -\int_\Omega p \operatorname{div} \bpsi \, dx \quad \text{for all }\ \bpsi\in H^1(\Omega)^d\text{ with }\bpsi|_{\Gv{1}}=0, \no \cdot \bpsi|_{\Gv{2}}=0.
\end{equation*}
Hence $(\mathbf{v},p)$ solve \eqref{eq:weak2phase}.
Moreover, we obtain the estimate
\begin{equation}
\left\Vert (\mathbf{v},p)\right\Vert _{H^{1}(\Omega)\times L^{2}(\Omega)}\leq C\left(\left\Vert \mathbf{f}\right\Vert _{L^{2}(\Omega)}+\left\Vert \mathbf{a}\right\Vert _{H^{\frac{1}{2}}(\Gamma_{t})}+\left\Vert \mathbf{g}\right\Vert _{H^{\frac{1}{2}}(\partial\Omega)}\right).\label{eq:2weak2estimate}
\end{equation}
We now show higher regularity of $\left(\mathbf{v},p\right)$ by localization.

Let $\eta^{\pm}\in C^{\infty}\left(\overline{\Omega}\right)$ be a
partition of unity of $\Omega$, such that the inclusions $\Omega^{+}(t)\cup\Gamma_{t}\left(\delta\right)\subset\left\{ \left.x\in\Omega\right|\eta^{+}(x)=1\right\} $
and $\partial\Omega\left(\delta\right)\subset\left\{ \left.x\in\Omega\right|\eta^{-}(x)=1\right\} $
hold. We choose $\eta^\pm$ such that $\{x\in\Omega: \eta^\pm(x)=1\}$ has smooth boundary and define $U^{\pm}:=\text{supp}\left(\eta^{\pm}\right)$, $\partial U_{0}^{-}:=\partial U^{-}\backslash\partial\Omega$
and 
\[
\dot{U}:=\left\{ \left.x\in\Omega\right|\eta^{+}(x)\in\left(0,1\right)\right\} =\left\{ \left.x\in\Omega\right|\eta^{-}(x)\in\left(0,1\right)\right\} .
\]
Moreover, we set $\dot{p}^{-}:=p\eta^{-}$ and $\tilde{\mathbf{v}}^{-}:=\mathbf{v}\eta^{-}$
in $\Omega$ and we correct the divergence of $\tilde{\mathbf{v}}^{-}$
with the help of the Bogovskii-operator: Let $\varphi\in C_{c}^{\infty}(\Omega)$
with $\text{supp}\left(\varphi\right)\subset U^{+}\backslash\dot{U}$
and $\int_{\Omega}\varphi\sd x=1$ and set 
\[
\hat{g}:=\operatorname{div}\left(\tilde{\mathbf{v}}^{-}\right)-\varphi\int_{U^{+}}\operatorname{div}\left(\tilde{\mathbf{v}}^{-}\right)\sd x
\]
in $U^{+}$. As $\mathbf{v}\in V(\Omega)$,
we have $\operatorname{div}\left(\tilde{\mathbf{v}}^{-}\right)=\mathbf{v}\cdot\nabla\eta^{-}$
and thus $\hat{g}\in H_{0}^{1}\left(U^{+}\right)$, $\int_{U^{+}}\hat{g}\sd x=0$.
Consequently, \cite[Theorem III.3.3]{galdi} implies that
there is $\hat{\mathbf{v}}^{-}\in H_{0}^{2}\left(U^{+}\right)$, which
we extend onto $\Omega$ by $0$, satisfying
\begin{align}
\operatorname{div}\hat{\mathbf{v}}^{-} & =\hat{g}\text{ in }U^{+},\nonumber \\
\left\Vert \hat{\mathbf{v}}^{-}\right\Vert _{H^{2}(\Omega)} & \leq C\left\Vert \mathbf{v}\right\Vert _{H^{1}(\Omega)}.\label{eq:vhut-}
\end{align}
Therefore, $\dot{\mathbf{v}}^{-}:=\tilde{\mathbf{v}}^{-}-\hat{\mathbf{v}}^{-}$
fulfills $\operatorname{div}\dot{\mathbf{v}}^{-}=0$ in $U^{-}$ since $\varphi\equiv0$
in that domain. Let now
$$
\bpsi\in\left\{ \mathbf{w}\in H^{1}(U^{-})^d: \mathbf{w}=0\text{ on }\partial U_{0}^{-}, \mathbf{w}|_{\Gv{1}}=0, \no\cdot \mathbf{w}|_{\Gv{2}}=0\right\},
$$
then 
\begin{align*}
\int_{U^{-}} & 2D_{s}\dot{\mathbf{v}}^{-}:D_{s}\bpsi-\dot{p}^{-}\operatorname{div}\bpsi\sd x+\int_{\Gv{2}}\alpha_{2}\mathbf{v}\cdot\bpsi\sd\mathcal{H}^{d-1}(s)+\int_{\Gv{3}}\alpha_{3}\mathbf{v}\cdot\bpsi\sd\mathcal{H}^{d-1}(s)\\
 & =\int_{U^{-}}2D_{s}\tilde{\mathbf{v}}^{-}:D_{s}\bpsi-p\operatorname{div}\left(\bpsi\eta^{-}\right)+\left(p\nabla\eta^{-}\right)\cdot\bpsi\sd x+\int_{\Gv{2}}\alpha_{2}\mathbf{v}\cdot\bpsi\sd\mathcal{H}^{d-1}(s)\\
 & \quad+\int_{\Gv{3}}\alpha_{3}\mathbf{v}\cdot\bpsi\sd\mathcal{H}^{d-1}(s)-\int_{U^{-}}2D_{s}\hat{\mathbf{v}}^{-}:D_{s}\bpsi\sd x\\
 & =\int_{U^{-}}\mathbf{f}\cdot\bpsi\eta^{-}\sd x-\int_{\partial\Omega}\mathbf{g}\cdot\bpsi\sd\mathcal{H}^{d-1}(s)+\left(p\nabla\eta^{-}\right)\cdot\bpsi\sd x\\
 & \quad+\int_{U^{-}}2\operatorname{div}\left(D_{s}\hat{\mathbf{v}}\right)\cdot\bpsi+\left(2D_{s}\mathbf{v}\nabla\eta^{-}-\operatorname{div}\left(\mathbf{v}\otimes\nabla\eta^{-}+\nabla\eta^{-}\otimes\mathbf{v}\right)\right)\cdot\bpsi\sd x,
\end{align*}
where we used the definition of $\dot{\mathbf{v}}^{-}$ and $\dot{p}^{-}$
in the first equality and integration by parts together with $\hat{\mathbf{v}}^{-}\in H_{0}^{2}\left(U^{+}\right)$
and $\nabla\eta^{-}=0$ on $U^{-}$ in the second equality. Additionally,
we employed the fact that $\left(\mathbf{v},p\right)$ is the weak
solution to (\ref{eq:weak2phase}). Hence, $\left(\dot{\mathbf{v}}^{-},\dot{p}^{-}\right)$
are a weak solution to the system
\begin{align}
-\Delta\dot{\mathbf{v}}^{-}+\nabla\dot{p}^{-} & =\tilde{\mathbf{f}} &  & \text{in }U^{-},\nonumber \\
\operatorname{div}\dot{\mathbf{v}}^{-} & =0 &  & \text{in }U^{-},\nonumber \\
\dot{\mathbf{v}}^{-} & =\hat{\mathbf{v}}^{-} &  & \text{on }\partial U_{0}^{-},\nonumber \\
  B_j(\mathbf{v}^{-},\dot{p}^{-}) & =\mathbf{g} &  & \text{on }\Gv{j}, j=1,2,3,\label{eq:onephase outer}
\end{align}
where 
\[
  \tilde{\mathbf{f}}:=\mathbf{f}\eta^-
  + p\nabla\eta^{-}+2\operatorname{div}(D_{s}\hat{\mathbf{v}})+2D_{s}\mathbf{v}\nabla\eta^{-}-\operatorname{div}\left(\mathbf{v}\otimes\nabla\eta^{-}+\nabla\eta^{-}\otimes\mathbf{v}\right)\in L^{2}(U^{-})
\]
and $\hat{\mathbf{v}}^{-}\in H^{\frac{3}{2}}(\partial U_{0}^{-})$
by the properties of the trace operator. Writing
\begin{equation*}
  \tilde{\mathbf{g}}:=
  \begin{cases}
\mathbf{g} &\text{on }\Gv{1},\\
\alpha_{j}\dot{\mathbf{v}}^{-}+\mathbf{g} &\text{on }\Gv{j},j=2,3,      
  \end{cases}
\end{equation*}
using localization techniques and results for strong solutions of
the stationary Stokes equation in one phase with inhomogeneous do-nothing
boundary condition (cf.\ Theorem 3.1 in \cite{ShimizuStokes}),
with Dirichlet boundary condition (cf.\ \cite{galdi}) and slip-boundary conditions (cf.\  Solonnikov and \v{S}\v{c}adilov~\cite[Theorem~2]{SolonnikovNavier}), 
we find that
$\left(\dot{\mathbf{v}}^{-},\dot{p}^{-}\right)\in H^{2}(U^{-})\times H^{1}(U^{-})$.
Moreover, regarding (\ref{eq:vhut-}), (\ref{eq:2weak2estimate})
and the definition of $\tilde{\mathbf{f}}$, we get 
\[
\left\Vert \left(\dot{\mathbf{v}}^{-},\dot{p}^{-}\right)\right\Vert _{H^{2}\left(U^{-}\right)\times H^{1}\left(U^{-}\right)}\leq C\left(\left\Vert \mathbf{f}\right\Vert _{L^{2}(\Omega)}+\left\Vert \mathbf{a}\right\Vert _{H^{\frac{1}{2}}\left(\Gamma_{t}\right)}+\left\Vert \mathbf{g}\right\Vert _{H^{\frac{1}{2}}\left(\partial\Omega\right)}\right).
\]
Analogously, we define $\tilde{\mathbf{v}}^{+}:=\mathbf{v}\eta^{+}$
and $\hat{\mathbf{v}}^{+}\in H_{0}^{2}(\dot{U})$ as a
solution to $\operatorname{div}\mathbf{\hat{\mathbf{v}}}^{+}=\operatorname{div}\tilde{\mathbf{v}}^{+}$.
Here, we do not need to correct the mean value, since 
\[
\int_{\dot{U}}\operatorname{div}\tilde{\mathbf{v}}^{+}\sd x=\int_{\partial\dot{U}}\mathbf{v}\cdot\mathbf{n}_{\partial\dot{U}}\eta^{+}\sd\mathcal{H}^{d-1}(s)=-\int_{\left\{ \eta^{+}=1\right\} }\operatorname{div}\mathbf{v}\sd x=0.
\]
We set $\dot{\mathbf{v}}^{+}:=\tilde{\mathbf{v}}^{+}-\hat{\mathbf{v}}^{+}$
and $\dot{p}^{+}:=p\eta^{+}$ and get after similar calculations as
before that $\left(\dot{\mathbf{v}}^{+},\dot{p}^{+}\right)$ is a
weak solution to the two phase stationary Stokes system 
\begin{align}
-\Delta\dot{\mathbf{v}}^{+}+\nabla\dot{p}^{+} & =\hat{\mathbf{f}} &  & \text{in }U^{+},\label{eq:zweiphasen1-1}\\
\operatorname{div}\dot{\mathbf{v}}^{+} & =0 &  & \text{in }U^{+},\\
\dot{\mathbf{v}}^{+} & =0 &  & \text{on }\partial U^{+},\\
\left[\dot{\mathbf{v}}^{+}\right] & =0 &  & \text{on }\Gamma_{t},\\
\left[2D_{s}\dot{\mathbf{v}}^{+}-\dot{p}^{+}\mathbf{I}\right]\mathbf{n}_{\Gamma_{t}} & =\mathbf{a} &  & \text{on }\Gamma_{t},\label{eq:zweiphasen5-1}
\end{align}
where $\hat{\mathbf{f}}\in L^{2}\left(U^{+}\right)$. Then \cite[Theorem 1.1]{shibatastokesdirichlet}
 implies $\dot{\mathbf{v}}^{+}|_{\Omega^{+}(t)}\in H^{2}(\Omega^{+}(t))$
and $\dot{\mathbf{v}}^{+}|_{U^{+}\backslash\Omega^{+}(t)}\in H^{2}(U^{+}\backslash\Omega^{+}(t))$,
and also that the pressure satisfies $\dot{p}^{+}|_{\Omega^{+}(t)}\in H^{1}(\Omega^{+}(t))$
and $\dot{p}^{+}|_{U^{+}\backslash\Omega^{+}(t)}\in H^{1}(U^{+}\backslash\Omega^{+}(t))$
with estimates in the associated norms. In particular, $\mathbf{v}=\dot{\mathbf{v}}^{+}$
in $\Omega^{+}(t)$ and $\mathbf{v}=\dot{\mathbf{v}}^{+}+\dot{\mathbf{v}}^{-}+\hat{\mathbf{v}}^{+}+\text{\ensuremath{\hat{\mathbf{v}}}}^{-}$
in $\Omega^{-}(t)$, yielding the desired regularity and
(\ref{eq:instoab}). To show that $C>0$ may be chosen independently
of $t\in\left[0,T_{0}\right]$, one may make use of extension arguments,
see e.g.\ the proof of  \cite[Lemma 2.10]{nsac}.
\end{proof}
\begin{theorem}
\label{scharfthe} Let $T\in\left(0,T_{0}\right]$. Let $\mathbf{b}\colon \Sigma\times\left[0,T\right]\rightarrow\mathbb{R}^d$,
$b\colon \Sigma\times\left[0,T\right]\rightarrow\mathbb{R}$, $a_{1}\colon \Omega\times\left[0,T\right]\rightarrow\mathbb{R}$,
$a_{2},a_{3},a_{5}\colon \Gamma\rightarrow\mathbb{R}$, $a_{4}\colon \partial\Omega\times[0,T]\rightarrow\mathbb{R}$,
$\mathbf{a}_{1}\colon \Omega\times\left[0,T\right]\rightarrow\mathbb{R}^d$,
$\mathbf{a}_{2},\mathbf{a}_{3},\mathbf{a}_{4},\mathbf{a}_{5}\colon\Gamma\rightarrow\mathbb{R}^d$
and $\mathbf{a}_{6}\colon \partial\Omega\times [0,T]\rightarrow\mathbb{R}^d$
be smooth given functions such that
\begin{equation*}
    \int_{\Gamma_t} \no_{\Gamma_t}\cdot \mathbf{a}_2 \, d\mathcal{H}^{d-1} + \int_{\partial\Omega} \no_{\partial\Omega}\cdot\mathbf{a}_6 \, d\mathcal{H}^{d-1} =0\qquad \text{if }\Gv{3}=\emptyset.
\end{equation*}
For every $g\in L^{2}\big(0,T;H^{\frac{1}{2}}(\Sigma)\big)$
and $h_{0}\in H^{2}(\Sigma)$ there exists a unique
solution $h\in X_{T}$ of 
\begin{align*}
D_{t,\Gamma}h+\mathbf{b}\cdot\nabla_{\Gamma}h-bh+\tfrac{1}{2}X_{0}^{*}\big((\mathbf{v}^{+}+\mathbf{v}^{-})\cdot\mathbf{n}_{\Gamma_{t}}\big)+\tfrac{1}{2}X_{0}^{*}\big(\big[\partial_{\mathbf{n}_{\Gamma_{t}}}\mu\big]\big) & =g &  & \text{in }\Sigma\times\left(0,T\right),\\
h\left(.,0\right) & =h_{0} &  & \text{in }\Sigma,
\end{align*}
where for every $t\in\left[0,T\right]$, the functions \textbf{$\mathbf{v}^{\pm}=\mathbf{v}^{\pm}(x,t)$},
$p^{\pm}=p^{\pm}(x,t)$ and $\mu^{\pm}=\mu^{\pm}(x,t)$
for $(x,t)\in\Omega_{T}^{\pm}$ with $\mathbf{v}^{\pm}\in H^{2}(\Omega^{\pm}(t))$,
$p^{\pm}\in H^{1}(\Omega^{\pm}(t))$ and $\mu^{\pm}\in H^{2}(\Omega^{\pm}(t))$
are the unique solutions to 
\begin{align}
\Delta\mu^{\pm} & =a_{1} &  & \text{in }\Omega^{\pm}(t),\label{eq:laplmu}\\
  \mu^{\pm} & = X_{0}^{*,-1}\big(\sigma\Delta_{\Gamma}h\pm a_{2}h\big)+a_{3} &  & \text{on }\Gamma_{t},\\
  \no_{\partial\Omega}\cdot \nabla \mu^{-} & =a_{4} &  & \text{on }\Gmu{1},\\
  \mu^{-} & =a_{4} &  & \text{on }\Gmu{2},\\
-\Delta\mathbf{v}^{\pm}+\nabla p^{\pm} & =\mathbf{a}_{1} &  & \text{in }\Omega^{\pm}(t),\label{eq:stokes}\\
\operatorname{div}\mathbf{v}^{\pm} & =0 &  & \text{in }\Omega^{\pm}(t),\\
[\mathbf{v}] & =\mathbf{a}_{2} &  & \text{on }\Gamma_{t},\label{eq:stokesjmp}\\
\left[2D_{s}\mathbf{v}-p\mathbf{I}\right]\mathbf{n}_{\Gamma_{t}} & =X_{0}^{*,-1}\big(\mathbf{a}_{3}h+\mathbf{a}_{4}\Delta_{\Gamma}h+a_{5}\nabla_{\Gamma}h+\mathbf{a}_{5}\big) &  & \text{on }\Gamma_{t},\label{eq:stokesbdry}\\
B_j(\mathbf{v}^{-},p^{-})& =\mathbf{a}_{6} &  & \text{on }\Gv{j}, j=1,2,3.\label{eq:stokesoutbdry}
\end{align}
Moreover, if $g$, $h_{0}$ and $b$, $\mathbf{b}$, $a_{i}$, and
$\mathbf{a}_{j}$ are smooth on their respective domains for $i\in\left\{ 1,\ldots,5\right\}$, $j\in\left\{ 1,\ldots,6\right\}$,
then $h$ is smooth and $p^{\pm}$, $\mathbf{v}^{\pm}$ and $\mu^{\pm}$
are smooth on $\Omega^{\pm}(t)$. 
\end{theorem}
\begin{proof}
We show this by a perturbation argument. First of all note that we
may without loss of generality assume that $a_{1}$, $a_{3}$, $a_{4}$,
$\mathbf{a}_{1}$, $\mathbf{a}_{2}$, $\mathbf{a}_{5}$, $\mathbf{a}_{6}=0$
on their respective domains. The above system may be reduced to this
case by solving
\begin{align*}
\Delta\hat{\mu}^{\pm} & =a_{1} &  & \text{in }\Omega^{\pm}(t),\\
  \hat{\mu}^{\pm} & =a_{3} &  & \text{on }\Gamma_{t},\\
    \no\cdot \nabla \hat{\mu}^{-} & =a_{4} &  & \text{on }\Gmu{1},\\
  \hat{\mu}^{-} & =a_{4} &  & \text{on }\Gmu{2}
\end{align*}
with the help of standard elliptic theory and 
\begin{align*}
-\Delta\hat{\mathbf{v}}^{\pm}+\nabla\hat{p}^{\pm} & =\mathbf{a}_{1} &  & \text{in }\Omega^{\pm}(t),\\
\operatorname{div}\mathbf{\hat{v}}^{\pm} & =0 &  & \text{in }\Omega^{\pm}(t),\\
[\mathbf{\hat{\mathbf{v}}}] & =\mathbf{a}_{2} &  & \text{on }\Gamma_{t},\\
\left[2D_{s}\mathbf{\hat{v}}-p\right]\mathbf{n} & =\mathbf{a}_{5} &  & \text{on }\Gamma_{t},\\
B_j(\mathbf{\hat{v}}^{-},\hat{p}^{-}) & =\mathbf{a}_{6} &  & \text{on }\Gv{j}, j=1,2,3,
\end{align*}
with the help of Theorem \ref{stokesthe} and setting 
\[
\hat{g}=g-\tfrac{1}{2}X_{0}^{*}\left([\partial_{\mathbf{n}_{\Gamma_{t}}}\hat{\mu}]+(\hat{\mathbf{v}}^{+}+\hat{\mathbf{v}}^{-})\cdot\mathbf{n}_{\Gamma_{t}}\right).
\]
Now let $t\in\left[0,T\right]$, $h\in H^{\frac{7}{2}}(\Sigma)$
and let $\mathbf{v}_{h}^{\pm}\in H^{2}(\Omega^{\pm}(t))^d$,
$p_{h}^{\pm}\in H^{1}(\Omega^{\pm}(t))$ be
the solution to (\ref{eq:stokes})\textendash (\ref{eq:stokesoutbdry}).
Multiplying (\ref{eq:stokes}) by $\mathbf{v}_{h}^{\pm}$ and integrating
in $\Omega^{\pm}(t)$ together with integration by parts
and the consideration of the boundary values (\ref{eq:stokesbdry})
and (\ref{eq:stokesoutbdry}) allows us to deduce
\begin{align}
\int_{\Omega^{+}(t)} & 2|D_{s}\mathbf{v}_{h}^{+}|^{2}\sd x+\int_{\Omega^{-}(t)}2|D_{s}\mathbf{v}_{h}^{-}|^{2}\sd x+\sum_{j=2,3}\int_{\Gv{j}}\alpha_j|\mathbf{v}_{h}^{-}|^{2}\sd\mathcal{H}^{d-1}(s)\nonumber \\
 & =\int_{\Gamma_{t}}X_{0}^{*,-1}\big(\mathbf{a}_{3}h+\mathbf{a}_{4}\Delta_{\Gamma}h+a_{5}\nabla_{\Gamma}h\big)\cdot\mathbf{v}_{h}^{-}\sd\mathcal{H}^{d-1}(s).\label{eq:v+-test}
\end{align}
Hence, by \cite[Corollary 5.8]{korn}  and the continuity of the trace
we find
\begin{equation}
\Vert \mathbf{v}_{h}^{-}\Vert _{H^{1}(\Omega^{-}(t))}\leq C\Vert h\Vert _{H^{2}(\Sigma)}\label{eq:v-est}
\end{equation}
for $C$ independent of $h$ and $t$. \cite[Corollary~5.8]{korn}, 
also implies 
\[
\int_{\Omega^{+}(t)}2|D_{s}\mathbf{v}_{h}^{+}|^{2}\sd x+\int_{\Gamma_{t}}|\mathbf{v}_{h}^{+}|^{2}\sd\mathcal{H}^{d-1}(s)\geq C\Vert \mathbf{v}_{h}^{+}\Vert _{H^{1}\left(\Omega^{+}(t)\right)}^{2},
\]
 leading to 
\begin{equation}
\Vert \mathbf{v}_{h}^{+}\Vert _{H^{1}(\Omega^{-}(t))}\leq C\Vert h\Vert _{H^{2}(\Sigma)}\label{eq:v+est}
\end{equation}
due to $\mathbf{v}_{h}^{+}=\mathbf{v}_{h}^{-}$ on $\Gamma_{t}$,
(\ref{eq:v-est}) and (\ref{eq:v+-test}). Defining 
\[
\mathcal{B}(t)\colon H^{\frac{7}{2}}(\Sigma)\rightarrow H^{\frac{1}{2}}(\Sigma)\colon h\mapsto\mathcal{B}(t)h=\tfrac{1}{2}X_{0}^{*}\big((\mathbf{v}_{h}^{+}+\mathbf{v}_{h}^{-})\cdot\mathbf{n}_{\Gamma_{t}}\big),
\]
we may use (\ref{eq:v-est}) and (\ref{eq:v+est}) to confirm 
\[
\Vert \mathcal{B}(t)h\Vert _{H^{\frac{1}{2}}(\Sigma)}\le C\Vert h\Vert _{H^{2}(\Sigma)}
\]
for $C>0$ independent of $h$ and $t$. As $H^{\frac{7}{2}}(\Sigma)$
is dense in $H^{2}(\Sigma)$ we can extend $\mathcal{B}(t)$
to an operator $\mathcal{B}(t):H^{2}(\Sigma)\rightarrow H^{\frac{1}{2}}(\Sigma)$
and $H^{2}(\Sigma)$ is close to $H^{\frac{1}{2}}(\Sigma)$
compared with $H^{\frac{7}{2}}(\Sigma)$.

The existence of a unique solution $h\in X_{T}$ with the properties
stated in the theorem is now a consequence of Theorem~\ref{thm:Perturbation}.
Higher regularity may be shown by localization and e.g.\ the
usage of difference quotients.
\end{proof}

\section*{Acknowledgments}
 The authors acknowledge support by the SPP 1506 "Transport Processes
at Fluidic Interfaces" of the German Science Foundation (DFG) through the grant AB285/4-2. Moreover, we are grateful to the anonymous referee for the careful reading a previous version of the manuscript and many helpful comments.

\appendix
\section{Existence of a Pressure}
\begin{lemma}\label{lem:pressure}
  Let $F\in \{\bpsi\in H^1(\Omega)^d: \bpsi|_{\Gv{1}}=0, \no_{\partial\Omega}\cdot \bpsi|_{\Gv{2}}=0 \}\to \R$ be linear and bounded such that
  \begin{equation*}
    F(\bpsi)= 0 \quad \text{for all } \bpsi\in V(\Omega)=\{ \bpsi\in H^1(\Omega)^d: \operatorname{div} \bpsi=0, \bpsi|_{\Gv{1}}=0, \no_{\partial\Omega}\cdot \bpsi|_{\Gv{2}}=0\}.
  \end{equation*}
  Then there is a unique $p\in L^2(\Omega)$ with $\int_\Omega p\,dx =0$ if $\Gv{3}=\emptyset$ such that
   \begin{equation*}
    F(\bpsi)= -\int_\Omega p\operatorname{div}\bpsi \,dx \qquad \text{for all } \bpsi\in H^1(\Omega)^d\text{ with }\bpsi|_{\Gv{1}}=0, \no_{\partial\Omega}\cdot \bpsi|_{\Gv{2}}=0.
  \end{equation*}
\end{lemma}
\begin{proof}
  We will apply the closed range theorem. To this end let
  \begin{align*}
    X&= \{\bpsi\in H^1(\Omega)^d: \bpsi|_{\Gv{1}}=0, \no_{\partial\Omega}\cdot \bpsi|_{\Gv{2}}=0 \},\\
    Y&= \left\{g\in L^2(\Omega): \int_\Omega g(x)\,dx =0 \text{ if }\Gv{3}=\emptyset\right\}
  \end{align*}
  and consider
  \begin{equation*}
    T \colon X\to Y \colon \bpsi\mapsto -\operatorname{div} \bpsi.
  \end{equation*}
  Then $T$ is onto, which can be seen as follows: Let $g\in Y$.

  If $\Gv{3}\neq \emptyset$, then there is a unique solution $q\in H^1(\Omega)$ of
  \begin{alignat*}{2}
    \Delta q &= g &\qquad & \text{in }\Omega,\\
    q|_{\Gv{3}} &= 0 && \text{on }\Gv{3},\\
    \no_{\partial\Omega} \cdot \nabla q|_{\Gv{1}\cup\Gv{2}}& =0 &&\text{on }\Gv{1}\cup \Gv{2}.
  \end{alignat*}
  Moreover, using the solvability of the stationary Stokes equation with non-homogenoues Dirichlet boundary conditions, we find some $\mathbf{w}\in H^1(\Omega)^d$ with $\operatorname{div} \mathbf{w}=0$ and
  \begin{equation*}
    \mathbf{w}|_{\Gv{1}} = \nabla q|_{\Gv{1}\cup \Gv{2}},\qquad \mathbf{w}|_{\Gv{3}}=0.
  \end{equation*}
  Then $\bpsi = \mathbf{w}-\nabla q\in X$ with $-\operatorname{div}\bpsi= g$.

  If $\Gv{3}= \emptyset$, we have $\int_\Omega g(x)\, dx=0$ and  can use the well-known Bogovskii operator to obtain some $\bpsi\in H^1_0(\Omega)$ with $-\operatorname{div} \bpsi=g$.

  Now the closed range theorem implies that $T' \colon Y'\to X'$ is injective and
  \begin{equation*}
    \mathcal{R}(T') =\mathcal{N}(T)^\circ=\{F\in X':  F(\bpsi) =0 \text{ for all }\bpsi\in V(\Omega)\}.
  \end{equation*}
  This proves the statement of the lemma.
\end{proof}



\begin{thebibliography}{10}

\bibitem{nsac}
H.~Abels and Y.~Liu, {{Sharp Interface Limit for a Stokes/Allen-Cahn
  System}}, \emph{Archives for Rational Mechanics and Analysis} \textbf{229} (2018),
  no.~1, 417--502.

\bibitem{NSCH2}
H.~Abels and A.~Marquardt, {Sharp interface limit of a
  {S}tokes/{C}ahn-{H}illiard system, part {II}: Approximate solutions},
  preprint, arXiv:2003.14267.

\bibitem{NSCH1}
\bysame, {Sharp interface limit of a {S}tokes/{C}ahn-{H}illiard system,
  part {I}: Convergence result}, preprint, arXiv:2003.03139.

\bibitem{abelswilke}
H.~Abels and M.~Wilke, {{Well-Posedness and Qualitative Behaviour of
  Solutions for a Two-Phase Navier-Stokes/Mullins-Sekerka System}}, \emph{Interfaces
  and Free Boundaries} \textbf{15} (2013), 39--75.

\bibitem{korn}
G.~Alessandrini, A.~Morassi, and E.~Rosset, {{The linear constraint in
  Poincar}\'{e} {and Korn type inequalities}}, \emph{Forum Mathematicum} \textbf{20}
  (2006), no.~3.

\bibitem{abc}
N.~Alikakos, P.~Bates, and X.~Chen, {{Convergence of the Cahn-Hilliard
  Equation to the Hele-Shaw Model}}, \emph{Archive for Rational Mechanics and
  Analysis} \textbf{128} (1994), no.~2, 165--205.

\bibitem{arendt}
W.~Arendt, R.~Chill, S.~Fornaro, and C.~Poupaud, {{$L^p$}-{Maximal
  Regularity for Non-Autonomous Evolution Equations}}, \emph{Journal of Differential
  Equations} \textbf{237} (2007), no.~1, 1--26.

\bibitem{chenAC}
X.~Chen, D.~Hilhorst, and E.~Logak, {{Mass Conserving Allen-Cahn Equation
  and Volume Preserving Mean Curvature Flow}}, \emph{Interfaces and Free Boundaries}
  \textbf{12} (2010), 527--549.

\bibitem{galdi}
G.~P. Galdi, \emph{{An Introduction to the Mathematical Theory of the
  Navier-Stokes Equations: Steady-State Problems}}, second ed., Springer
  Monographs in Mathematics, 2011.

\bibitem{mclean}
William McLean, \emph{{Strongly Elliptic Systems and Boundary Integral
  Equations}}, Cambridge University Press, 2000.

\bibitem{pruess}
J.~Pruess and G.~Simonett, \emph{{Moving Interfaces and Quasilinear Parabolic
  Evolution Equations}}, Birkh{\"a}user Basel, 2016.

\bibitem{stefan}
Stefan Schaubeck, \emph{{Sharp Interface Limits for Diffuse Interface Models}},
  Ph.D. thesis, University of Regensburg, urn:nbn:de:bvb:355-epub-294622, 2014.

\bibitem{shibatastokesdirichlet}
Y.~Shibata and S.~Shimizu, {{On a Resolvent Estimate of the Interface
  Problem for the Stokes System in a Bounded Domain}}, \emph{Journal of Differential
  Equations} \textbf{191} (2003), no.~2, 408--444.

\bibitem{ShimizuStokes}
\bysame, {{On the Lp-Lq Maximal Regularity of the Neumann Problem for the
  Stokes Equations in a Bounded Domain}}, \emph{Journal f{\"u}r die reine und
  angewandte Mathematik} \textbf{615} (2007), 1--53.

\bibitem{SolonnikovNavier}
V.~A. Solonnikov and V.~E. {\v{S}}{\v{c}}adilov, {A certain boundary value
  problem for the stationary system of {N}avier-{S}tokes equations}, \emph{Trudy Mat.
  Inst. Steklov.} \textbf{125} (1973), 196--210, 235, Boundary value problems of
  mathematical physics, 8.

\end{thebibliography}

\providecommand{\bysame}{\leavevmode\hbox to3em{\hrulefill}\thinspace}
\providecommand{\MR}{\relax\ifhmode\unskip\space\fi MR }
\providecommand{\MRhref}[2]{%
  \href{http://www.ams.org/mathscinet-getitem?mr=#1}{#2}
}
\providecommand{\href}[2]{#2}

\end{document}